\documentclass[reqno]{amsart}
\usepackage{amscd}
\usepackage[dvips]{graphics}

\def\beq{\begin{equation}}
\def\eeq{\end{equation}}
\def\ba{\begin{array}}
\def\ea{\end{array}}

\def\R{\mathbb R}

\numberwithin{equation}{section}

\newenvironment{key words}{\textbf{Keywords}\mbox{  }}{ }
\newtheorem{theorem}{Theorem}[section]

\newtheorem{lemma}[theorem]{Lemma}
\renewenvironment{proof}{\noindent{\textbf{Proof.}}}{\hfill$\Box$}
\theoremstyle{remark}
\newtheorem{remark}[theorem]{\textbf{Remark}}
\theoremstyle{plain}

\makeatletter
    \newcommand{\rmnum}[1]{\romannumeral #1}
    \newcommand{\Rmnum}[1]{\expandafter\@slowromancap\romannumeral #1@}
\makeatother

\begin{document}
\title[Blowup analysis for subcritical integral equations]{\textbf{Blowup analysis for integral equations on bounded domains}}

\author  {Qianqiao Guo}

\address{Qianqiao Guo, Department of Applied Mathematics, Northwestern Polytechnical University, Xi'an, Shaanxi, 710129, China}

\email{gqianqiao@nwpu.edu.cn}

%\address{Meijun Zhu, Department of Mathematics,
%The University of Oklahoma, Norman, OK 73019, USA}

%\email{mzhu@math.ou.edu}

\date{}
\maketitle

\begin{abstract}
Consider the integral equation 
\begin{equation*}
f^{q-1}(x)=\int_\Omega\frac{f(y)}{|x-y|^{n-\alpha}}dy,\ \ f(x)>0,\quad x\in \overline \Omega,
\end{equation*}
where $\Omega\subset \mathbb{R}^n$ is a smooth bounded domain. For $1<\alpha<n$, the existence of energy maximizing positive solution in subcritical case $2<q<\frac{2n}{n+\alpha}$, and nonexistence of energy maximizing positive solution in critical case $q=\frac{2n}{n+\alpha}$ are proved in \cite{DZ2017}. For $\alpha>n$,  the existence of energy minimizing positive solution in subcritical case  $0<q<\frac{2n}{n+\alpha}$, and nonexistence of energy minimizing positive solution in critical case $q=\frac{2n}{n+\alpha}$ are also proved in \cite{DGZ2017}.
Based on these, in this paper, the blowup behaviour of energy maximizing positive solution as $q\to (\frac{2n}{n+\alpha})^+ $ (in the case of $1<\alpha<n$), and the blowup behaviour of energy minimizing positive solution as $q\to (\frac{2n}{n+\alpha})^-$ (in the case of $\alpha>n$) are analyzed. We see that for $1<\alpha<n$ the blowup behaviour obtained  is quite similar to that of the elliptic equation involving subcritical Sobolev exponent. But for $\alpha>n$, different phenomena appears. 
\end{abstract}

\begin{key words}
Blowup analysis, Integral equation, Hardy-Littlewood-Sobolev inequality, Reversed Hardy-Littlewood-Sobolev inequality
\end{key words}

{\small\bf 2010 Mathematics Subject Classification:} 45G05, 35B09, 35B44


\section{Introduction}
In this paper we consider the integral equation
\begin{equation*}
 (\mathcal{P}_q)\ \ \ \ \ \ \ \ \ \ \ \ \ \ \  \ \ \ \ \ \ \ \ \ \ \ \ \  f^{q-1}(x)=\int_\Omega\frac{f(y)}{|x-y|^{n-\alpha}}dy,\ \ f(x)>0,\quad x\in \overline \Omega,\ \ \ \ \ \ \ \ \ \ \ \ \ \ \ \ \ \
\end{equation*}
where $\Omega\subset \mathbb{R}^n$ is a smooth bounded domain. 

The equation $(\mathcal{P}_q)$ is quite similar to the classical scalar curvature equation, but with a global defined boundary condition. The study of this integral equation was called to our attention  by  Li in \cite{Li2004}, where he was studying the global defined integral equations.  This nonlocal equation is also much closer to the integral curvature equation introduced by Zhu in \cite{Zhu16}. In \cite{DZ2017} and \cite{DGZ2017}, the existence of extremal energy solution as well as the  nonexistence (on star-shaped domains) of positive solutions to $(\mathcal{P}_q)$  are studied  for $1 <\alpha<n$ and $\alpha>n$  respectively. In particular as $q$ going to the critical exponent $q_\alpha:=\frac{2n}{n+\alpha}$, the sequence of extremal energy solutions usually do not converge in $L^\infty$ sense.  In this paper, we would like to analyze how does the sequence of extremal energy solutions blow up in both cases {$1<\alpha<n$ and $\alpha>n$.

For simplicity, we denote $p_\alpha:=\frac{2n}{n-\alpha}$, the conjugate exponent of $q_\alpha$.

\subsection{For $1<\alpha<n$}

In this case the integral equation $(\mathcal{P}_q)$ is the Euler-Lagrange equation for the maximizer (if the supremum is attained) of
\begin{eqnarray*}
\xi_{\alpha,q}(\Omega):=\sup_{f\in L^q(\Omega),f\neq 0}\frac{\int_{\Omega} \int_{\Omega} f(x)|x-y|^{-(n-\alpha)} f(y) dx dy}{||f||^2_{L^q(\Omega)}}.
\end{eqnarray*}
Due to the classical sharp Hardy-Littlewood-Sobolev (HLS for short) inequality \cite{HL1928, HL1930, So1963, Lieb1983}, one can show (see, for example, Dou and Zhu \cite{DZ2017}) that
%$$ \xi_\alpha (\Omega)= N_\alpha $$
 $\xi_{\alpha,q_\alpha} (\Omega) $ is not attained by any functions in any smooth domain $\Omega \ne \mathbb{R}^n$, but  $\xi_{\alpha,q} (\Omega)$ is always attained by a maximizer $f\in C^{1}(\overline{\Omega}), f>0,$ in the subcritical case $q_\alpha<q<2$ in any smooth bounded domain $\Omega$.
  
This indicates that,  for smooth bounded domain $\Omega$, the integral equation $(\mathcal{P}_q)$ admits an energy maximizing positive solution (i.e., a positive solution which is also a maximizer to $\xi_{\alpha,q}(\Omega)$) $f_q\in C^{1}(\overline{\Omega})$ in the subcritical case $q_\alpha<q<2$, but does not admit  any energy maximizing positive solution in the critical case $q=q_\alpha$.
  
 Based on these, we claim that the energy maximizing positive solution $f_q$  must blow up as $q\to (q_\alpha)^+$, that is, $\max\limits_{x\in\overline{\Omega}}f_q\to \infty$  as $q\to (q_\alpha)^+$(see (\rmnum 1) of the following Theorem \ref{BL-alpha-small}). In this paper we will study the blowup behaviour of $f_q$ as $q\to (q_\alpha)^+$.

To compare with the semilinear elliptic equation, we denote $u_q(x):=f_q^{q-1}(x)$, which is a positive solution to 
\begin{equation}\label{HB-sub-1}
u(x)=\int_\Omega\frac{u^{p-1}(y)}{|x-y|^{n-\alpha}}dy,\ \ u(x)>0,\quad x\in \overline\Omega,
\end{equation}
where $\frac{1}{q}+\frac{1}{p}=1$. For convenience we also call $u_q$ an energy maximizing positive solution to \eqref{HB-sub-1}.
% if $f_q(x)=u_q^{p-1}(x)$ is an energy maximizing positive solution  to $(\mathcal{P}_q)$. 
Notice that $q_\alpha<q<2$ is equivalent to $2<p<p_\alpha$ and $q\to (q_\alpha)^+$ is equivalent to $p\to (p_\alpha)^-$ . The first result of this paper is as following.
\begin{theorem}\label{BL-alpha-small}
Let $\alpha \in (1, n)$ and $\Omega$ a smooth bounded domain.
For $q_\alpha<q<2$, if $u_q$ is an energy maximizing positive solution  to \eqref{HB-sub-1}, then as $q\to (q_\alpha)^+$, up to a subsequence, 

(\rmnum 1). $\max\limits_{x\in\overline\Omega}u_q(x):=u_q(x_q)\to\infty$, and $x_q$ will stay away from $\partial\Omega$.

(\rmnum 2). $u_q(x)\le C(\frac{\mu_q}{\mu_q^2+|x-x_q|^2})^\frac{n-\alpha}{2}$, where $\mu_q= u_q^{-\frac {p-2}{\alpha}}(x_q)$.

(\rmnum 3). $u_q(x_q)u_q(x) \to \frac{\sigma_{n,\alpha}}{|x-x_0|^{n-\alpha}}$ if $x\neq x_0,\ x\in\overline{\Omega}$, where $x_0$ is the unique point  such that $x_q\to x_0$,  and
$\sigma_{n,\alpha}=(\pi^{\frac{n}{2}}\frac{\Gamma(\frac{\alpha}{2})}{\Gamma(\frac{n+\alpha}{2})})^{\frac{\alpha-n}{\alpha}}$.
\end{theorem}

For the semi-linear elliptic equation
 \begin{equation}\label{Sob-critical-1}
  \begin{cases}
  -\Delta u=n(n-2)u^{p-1}, &\quad u>0, \quad \text{in}\quad \Omega,\\
  u=0,&\quad\text{on}\quad \partial\Omega,
  \end{cases}
  \end{equation}
where $\Omega\subset \mathbb{R}^n (n\ge 3)$ is a smooth bounded domain, $p\in(2,2^*)$, $2^*:=\frac{2n}{n-2}$ is the critical Sobolev exponent,
the blowup behaviour of the extremal energy positive solutions as $p\to (2^*)^-$ (i.e., the sequence of positive solutions which is a minimizing sequence for the Sobolev inequality) has been studied extensively, see for example \cite{BP1989,FluWei97MM,Han,Rey89MM,Rey90JFA}. 
Comparing with the semi-linear elliptic equation, no accurate form of the Green's function corresponding to the integral equation \eqref{HB-sub-1} can be used directly, which is the big difference. Again due to this difference, unlike what has been proved by Flucher and Wei \cite{FluWei97MM} for the semi-linear elliptic equation \eqref{Sob-critical-1}, the location of the blowup point $x_0$ is still not clear in Theorem \ref{BL-alpha-small}, which is our next concern. 
Another difference comes from the nonlocal property of the integral equation \eqref{HB-sub-1}. Thus some different techniques are needed even we follow the line of \cite{Han}.

%---------------------------------------------------------------
%---------------------------------------------------------------
\subsection{For $\alpha>n$} In this case the integral equation $(\mathcal{P}_q)$ is the Euler-Lagrange equation for the mimimizer (if the infimum is attained) of
\begin{eqnarray*}
\widehat{\xi}_{\alpha,q}(\Omega)=\inf_{f\in L^q(\Omega),f\ge 0,f\neq 0}\frac{\int_{\Omega} \int_{\Omega} f(x)|x-y|^{-(n-\alpha)} f(y) dx dy}{||f||^2_{L^q(\Omega)}},
\end{eqnarray*}
where we still denote $L^q(\Omega):=\{f| \int_{\Omega}|f|^q dx<\infty\}$, and $||f||_{L^q(\Omega)}:=(\int_{\Omega}|f|^q dx)^{\frac{1}{q}}$ even it is not a norm for $0<q\le q_\alpha<1$.

Due to the sharp  reversed HLS inequality \cite{DZ2015}, we have proved in \cite{DGZ2017} that
    $\xi_{\alpha,q_\alpha} (\Omega)$ is not attained by any functions in smooth domain $\Omega \ne \mathbb{R}^n$,  and $\xi_{\alpha,q} (\Omega)$ is always attained by a minimizer $f\in C^{1}(\overline{\Omega}), f>0,$ in the subcritical case $0<q<q_\alpha$ in any smooth bounded domain $\Omega$. That is, for smooth bounded domain $\Omega$, the integral equation $(\mathcal{P}_q)$ admits an energy minimizing positive solution (i.e., a positive solution which is a minimizer to $\widehat{\xi}_{\alpha}(\Omega)$) $f_q\in C^{1}(\overline{\Omega})$ in the subcritical case  $0<q<q_\alpha$, but admits no any energy minimizing positive solution in the critical case $q=q_\alpha$.
    
  Hence, for $\alpha>n$, we again claim that the energy minimizing positive solution $f_q$ to $(\mathcal{P}_q)$ must blow up as $q\to (q_\alpha)^-$, that is, $\max\limits_{x\in\overline{\Omega}}f_q\to +\infty$  as $q\to (q_\alpha)^-$ (see (\rmnum 1) of the following Theorem \ref{BL-alpha-big}). Then it is interesting to study the blowup behaviour of the energy minimizing positive solution $f_q$ to $(\mathcal{P}_q)$ as $q\to (q_\alpha)^-$.

Denote $u_q(x)=f_q^{q-1}(x)$. Then $u_q$ satisfies \eqref{HB-sub-1} and, for convenience, we also call $u_q$ an energy minimizing positive solution. Notice that $0<q<q_\alpha$ is equivalent to $p_\alpha<p<0$ and $q\to (q_\alpha)^-$ is equivalent to $p\to (p_\alpha)^+$. 

We point out that when $\alpha>n$ the integral equation \eqref{HB-sub-1} is of negative power. We have the following analogue of Theorem \ref{BL-alpha-small}.
\begin{theorem}\label{BL-alpha-big}
Let $\alpha>n$ and $\Omega$ a smooth bounded domain. For $0<q<q_\alpha$, if $u_q$ is an energy minimizing positive solution to \eqref{HB-sub-1}, then as $q\to (q_\alpha)^-$, up to a subsequence,

(\rmnum 1). $\min\limits_{x\in\overline\Omega}u_q(x):=u_q(x_q)\to 0$, and $x_q$ will stay away from $\partial\Omega$.

(\rmnum 2).  $u_q(x)\ge C(\frac{\mu_q}{\mu_q^2+|x-x_q|^2})^\frac{n-\alpha}{2}$, where $\mu_q= u_q^{-\frac {p-2}{\alpha}}(x_q)$.

(\rmnum 3). $u_q(x_q)u_q(x) \to \frac{\sigma_{n,\alpha} }{|x-x_0|^{n-\alpha}}$ if $x\neq x_0,\  x\in\overline{\Omega}$, where $x_0\in\Omega$ is the unique point such that $x_q\to x_0$.

\end{theorem}

 For $n=1, \alpha=2$, similar blowup analysis for the following semilinear elliptic equation with negative power 
\begin{equation*}
4 u_{\theta \theta}+ u=R(\theta){u^{-3+\epsilon}}, \ \ u>0, \ \ \mbox{on} \ \  \mathbb{S}^1
\end{equation*} 
 was carried out in \cite{JWW11} and \cite{GZ16}, where $0<\epsilon<2$.  As in \cite{GZ16} we call $x_0$ in Theorem \ref{BL-alpha-big} the blowup point of $u_q$, which is actually the most important point since $u_q(x_0)\to 0$ and $u_q(x)\to \infty$ for $x\neq x_0, x\in\overline{\Omega},$  as $q\to (q_\alpha)^-$. We can see that for  $\alpha>n$, the negative power does bring some differences 
(for example, no Nash-Moser iteration can be used) comparing with the case $1<\alpha<n$ in carrying out the blowup analysis.

\section{Blowup behaviour for $\alpha\in(1,n)$}

\noindent
{\bf Proof of (\rmnum1) of Theorem \ref{BL-alpha-small}}. Let $u_q(x_q):=\max\limits_{\overline \Omega} u_q(x)$. We first prove $u_q(x_q) \to \infty$ as $q \to (q_\alpha)^+$. By contrary, we assume $u_q(x) \le C$ uniformly, then by the results in \cite{DZ2017} it is easy to see that the $C^1$ norm of $u_q(x)$ is also uniformly bounded, thus $u_q(x)$ is equicontinuous. Then we conclude that $u_q(x) \to u^*(x)\ge 0$ pointwise as $q \to (q_\alpha)^+$ and $u^*(x)$ is a nonnegative solution to  \eqref{HB-sub-1}. So $f^*=(u^*)^{p_\alpha-1}$  is a nonnegative solution to  $(\mathcal{P}_{q_\alpha})$. Notice that $f_q=u_q^{p-1}$ is the energy maximizing positive solution to $(\mathcal{P}_q)$, and  $\xi_{\alpha,q}(\Omega)\to \xi_{\alpha,q_\alpha}(\Omega)>0$ as $q\to (q_\alpha)^+$ (see \cite{DZ2017}). Then  
\begin{equation}\label{Bas-equ1}
||f_q||_{L^q(\Omega)}=(\xi_{\alpha,q} (\Omega))^{\frac{1}{q-2}}\ge\frac{(\xi_{\alpha,q_\alpha}(\Omega))^{\frac{1}{q_\alpha-2}}}{2}>0
\end{equation}
as $q$ close to $q_\alpha$ and $||f_q||_{L^q(\Omega)}\to ||f^*||_{L^{q_\alpha}(\Omega)}$  as $q \to (q_\alpha)^+$. So $f^*\not\equiv 0$ and thus is positive pointwise, and $f^*$ is an energy maximizing positive solution to $(\mathcal{P}_{q_\alpha})$, which is obviously of $C^1$. Hence we obtain a contradiction.

So we have $u_q(x_q)\to \infty,$ and
$x_q\to x_0\in\overline\Omega$, up to a subsequence.

Now we prove that $x_q$ will stay away from $\partial\Omega$ as $q\to (q_\alpha)^+$. For simplicity, in this part below we write $u$ instead of $u_q$,  $f$ instead of $f_q$.

(\Rmnum1). The domain $\Omega$ is strictly convex.

By using the method of moving planes to integral equation \eqref{HB-sub-1}, which is omitted here since it is similar to Theorem 3.4 in \cite{DZ2017} (see also \cite{CFY2014}), we can prove that there exist $t_0>0, \alpha>0$ depending on the domain only, such that for every $x\in\partial\Omega$, $u(x-t\nu)$ is increasing in $t\in [0, t_0]$, where $\nu\in \mathbb{R}^n, |\nu|=1$ satisfying $(\nu,\overrightarrow{n}(x))\ge\alpha$, and$\overrightarrow{n}(x)$ is the unit outer normal of $\Omega$ at the boundary point $x$.

Now as in \cite{Han}, we know that there are $\gamma,\delta>0$ only depending on the domain $\Omega$ such that for any $x\in \{z\in\overline{\Omega}: d(z,\partial\Omega)<\delta \}$, there exists a measurable set $\Gamma_x\subset\{z\in\overline{\Omega}: d(z,\partial\Omega)>\delta/2 \}$ satisfying $meas(\Gamma_x)\ge\gamma$ and $u(y)\ge u(x)$ for any $y\in\Gamma_x$. In fact, $\Gamma_x$ can be taken to be a piece of cone with vertex at $x$. Then for any $x\in \{z\in\overline{\Omega}: d(z,\partial\Omega)<\delta \}$, by \eqref{Bas-equ1},
\begin{eqnarray*}
u(x)&\le&\frac{1}{meas(\Gamma_x)}\int_{\Gamma_x} u(y)dy\le \gamma^{-1}|\Omega|^{\frac{1}{q}}(\int_{\Omega} u^p(y)dy)^{\frac{1}{p}}\\
&=&\gamma^{-1}|\Omega|^{\frac{1}{q}}(\int_{\Omega} f^q(y)dy)^{\frac{1}{p}}\le C<\infty
\end{eqnarray*}
uniformly, which implies that $x_q$ must stay out of the region $\{z\in\overline{\Omega}: d(z,\partial\Omega)<\delta \}$ as $q\to (q_\alpha)^+$.

(\Rmnum2). The domain $\Omega$ is not necessarily strictly convex.

As in \cite{Han}, we use the Kelvin transformation at each boundary point and then apply the method of moving planes. We show the details of the argument here for the reader's convenience since it is a little different from Theorem 3.4 in \cite{DZ2017}.

Pick any point $P\in\partial\Omega$. For simplicity, we assume the ball $B(0,1)$ contacts $P$ from the exterior of $\Omega$. Let $w$ be the Kelvin transform of $u$, that is,
$$w(x)=\frac{1}{|x|^{n-\alpha}}u(\frac{x}{|x|^2}), \ \ \ \ x\in\overline{\Omega^*},$$
where $\Omega^*$ is the image of $\Omega$ under the Kelvin transform. Then $w(x)$ satisfies
\begin{equation}\label{UB-7-add}
w(x)=\int_{\Omega^*}\frac{w^{p-1}(y)}{|x-y|^{n-\alpha}}|y|^{(n-\alpha)(p-p_\alpha)} dy,\ \ \ x\in\overline{\Omega^*}.
\end{equation}
Without loss of generality we assume $P=(-1,0,\cdots,0)$ and $x_1=-1$ is the tangent plane of $\partial\Omega^*$ at $P$.  Then it is enough to prove that $w(x)$ is increasing along the $x_1$ direction in a neighbourhood of $P$.

Assume that we can move the tangent plane at $P$ along the $x_1$ direction to the limiting place $x_1=\overline{x}<0$, denoted by $T_0$, such that the reflection of $\Omega^*\cap \{x_1<\lambda\}$ with respect to $T_\lambda$ is also a subset of $\Omega^*$, where $T_\lambda=\{x\in \mathbb{R}^n \ | \ x_1=\lambda\}$, $\lambda\in (-1,\overline{x}]$. 

We denote by $\Omega_1^\prime$ the reflection of $\Omega_1:=\Omega^*\cap \{x_1<\overline{x}\}$ with respect to $T_0$.  Now we apply the method of moving planes to integral equation \eqref{UB-7-add} on $\Omega^\prime:=\Omega_1\cup \Omega_1^\prime$.

For any real number $\lambda\in (-1,\overline{x})$, define $x^\lambda=(2\lambda-x_1,x_2,\cdots,x_n)$ as the reflection of point $x=(x_1,x_2,\cdots,x_n)$ about the plane $T_\lambda$. Let
\begin{eqnarray*}
\Sigma_\lambda:=\{x=(x_1,x_2,\cdots,x_n)\in \Omega^\prime\ | \ -1<x_1<\lambda\},
\end{eqnarray*}
 $\widetilde{\Sigma}_\lambda:=\{x^\lambda~|~x\in{\Sigma}_\lambda\}$ be the reflection of $\Sigma_\lambda$ about the plane $T_\lambda$, and $\Sigma^C_\lambda=\Omega^*\backslash\Sigma_\lambda$ be the complement of $\Sigma_\lambda$ in $\Omega^*$. Set $w_\lambda(x):=w(x^\lambda).$
We shall  complete the proof  in two steps. 

\noindent \textbf{Step 1.}  We show that for $\lambda$ larger than and sufficiently close to $-1$,
\begin{equation}\label{UB-6-add}
w_\lambda(x)-w(x)\geq 0,  \quad \forall x\in \Sigma_\lambda,
\end{equation}
which can be obtained easily by  
\begin{eqnarray*}
\frac{\partial w(x)}{\partial x_1}|_{x_1=-1}
=(\alpha-n)\int_{\Omega^*}\frac{w^{p-1}(y)}{|x-y|^{n-\alpha+2}}(-1-y_1)|y|^{(n-\alpha)(p-p_\alpha)} dy>0
\end{eqnarray*}
since $w(x)\in C^1(\overline{\Omega^*})$.

\noindent \textbf{Step 2.} Plane $T_\lambda$ can be moved continuously towards right to its limiting position $T_0$ as long as inequality \eqref{UB-6-add} holds. Thus we  conclude that $w(x)$ is increasing in the $x_1$ direction for any $x\in\Omega_1$.

Define
\[\lambda_0:=\sup\{\lambda\in (-1,\overline{x})\ | \ w_\mu(y)\geq w(y),\forall y \in\Sigma_\mu, \ \ -1<\mu\leq\lambda\}.\]
We claim that $\lambda_0=\overline{x}$.
We prove it by contradiction. Suppose not, that is,  $\lambda_0<\overline{x}$.

We first  show that
\[w_{\lambda_0}(x)-w(x)>0\]
in the interior of $\Sigma_{\lambda_0}$.

In fact, since $|x-y|<|x-y^{\lambda_0}|$ for $x, y\in{\Sigma}_{\lambda_0}$, we have
\begin{eqnarray}\nonumber
&&w_ {\lambda_0}(x)-w(x)\\\nonumber
&=&\int_{\Sigma_{\lambda_0}}\big[\frac1{|x^{\lambda_0}-y|^{n-\alpha}}-\frac1{|x-y|^{n-\alpha}}\big]{w^{p-1}(y)}|y|^{(n-\alpha)(p-p_\alpha)}dy\\\nonumber
&&+\int_{\widetilde{\Sigma}_{\lambda_0}}\big[\frac1{|x^{\lambda_0}-y|^{n-\alpha}}-\frac1{|x-y|^{n-\alpha}}\big]{w^{p-1}(y)}|y|^{(n-\alpha)(p-p_\alpha)}dy\\\nonumber
&&+\int_{\Sigma^C_{\lambda_0}\backslash\widetilde{\Sigma}_{\lambda_0}}\big[\frac1{|x^{\lambda_0}-y|^{n-\alpha}}-\frac1{|x-y|^{n-\alpha}}\big]{w^{p-1}(y)}|y|^{(n-\alpha)(p-p_\alpha)}dy\\\nonumber
&=&\int_{\Sigma_{\lambda_0}}\big[\frac1{|x-y^{\lambda_0}|^{n-\alpha}}-\frac1{|x-y|^{n-\alpha}}\big]{w^{p-1}(y)}|y|^{(n-\alpha)(p-p_\alpha)}dy\\\nonumber
&&-\int_{\Sigma_{\lambda_0}}\big[\frac1{|x-y^{\lambda_0}|^{n-\alpha}}-\frac1{|x-y|^{n-\alpha}}\big]{w^{p-1}_{\lambda_0}(y)}|y_{\lambda_0}|^{(n-\alpha)(p-p_\alpha)}dy\\\nonumber
&&+\int_{\Sigma^C_{\lambda_0}\backslash\widetilde{\Sigma}_{\lambda_0}}\big[\frac1{|x^{\lambda_0}-y|^{n-\alpha}}-\frac1{|x-y|^{n-\alpha}}\big]{w^{p-1}(y)}|y|^{(n-\alpha)(p-p_\alpha)}dy\\\nonumber
&\ge&\int_{\Sigma_{\lambda_0}}\big[\frac1{|x-y|^{n-\alpha}}-\frac1{|x-y^{\lambda_0}|^{n-\alpha}}\big]({w^{p-1}_{\lambda_0}(y)}-{w^{p-1}(y)})|y|^{(n-\alpha)(p-p_\alpha)}dy\\\nonumber
&&+\int_{\Sigma^C_{\lambda_0}\backslash\widetilde{\Sigma}_{\lambda_0}}\big[\frac1{|x^{\lambda_0}-y|^{n-\alpha}}-\frac1{|x-y|^{n-\alpha}}\big]{w^{p-1}(y)}|y|^{(n-\alpha)(p-p_\alpha)}dy\nonumber\\
&\geq&\int_{\Sigma^C_{\lambda_0}\backslash\widetilde{\Sigma}_{\lambda_0}}\big[\frac1{|x^{\lambda_0}-y|^{n-\alpha}}-\frac1{|x-y|^{n-\alpha}}\big]{w^{p-1}(y)}|y|^{(n-\alpha)(p-p_\alpha)}dy. \label{UB-10-add}
\end{eqnarray}
If there exists some point $\xi\in\Sigma_{\lambda_0}$ such that $w(\xi)= w_{\lambda_0}(\xi)$, then,  since $|x-y|>|x^{\lambda_0}-y|$ for $x\in{\Sigma}_{\lambda_0},y\in\Sigma^C_{\lambda_0}$, we deduce  from \eqref{UB-10-add} that
\[w(y)\equiv 0,\quad \forall y\in\Sigma^C_{\lambda_0}\backslash\widetilde{\Sigma}_{\lambda_0}.\]
This contradicts to the  assumption that $w>0$ since $\Sigma^C_{\lambda_0}\backslash\widetilde{\Sigma}_{\lambda_0}$ is not empty.  Hence $w_{\lambda_0}(x)-w(x)>0$ for all $ x \in \Sigma_{\lambda_0}.$

For any $\delta>0$, choose $0<\lambda(\delta)<\lambda_0$ but close to $\lambda_0$  such that $meas(\Sigma_{\lambda_0}\backslash\Sigma_{\lambda(\delta)})<\delta$.  There is a number $C(\delta)>0$, such that
  \[w_{\lambda_0}(x)-w(x)\geq C(\delta)>0,\  \  \  \forall x\in \Sigma_{\lambda(\delta)}.\]
  Thus, there exists $\varepsilon_1>0$ such that $\forall\lambda\in [\lambda_0,\lambda_0+\varepsilon_1)$,
\[w_{\lambda}(x)-w(x)\geq \frac{C(\delta)}2>0, \ \  \forall x\in \Sigma_{\lambda(\delta)}.\]
We can further  assume that $\varepsilon_1$ small enough so that $meas(\Sigma_{\lambda_0+\varepsilon_1}\backslash \Sigma_{\lambda(\delta)})\le 2\delta<1$.

It is easy to see
\[\Sigma^w_\lambda:=\{x\in \Sigma_\lambda  \ | \ w(x)>w_\lambda(x)\}\subset\Sigma_{\lambda_0+\varepsilon_1}\backslash \Sigma_{\lambda(\delta)},\quad ~\forall\lambda\in [\lambda_0,\lambda_0+\varepsilon_1).
\]
Similar to \eqref{UB-10-add} and by using the mean value theorem, we have:  $\forall x\in \Sigma^w_\lambda$,
\begin{eqnarray}\nonumber
0&<&w(x)-w_\lambda(x)\nonumber\\\nonumber
&\le &\int_{\Sigma_{\lambda}}\big[\frac1{|x-y|^{n-\alpha}}-\frac1{|x^{\lambda}-y|^{n-\alpha}}\big]{w^{p-1}(y)}|y|^{(n-\alpha)(p-p_\alpha)}dy\\\nonumber
&&+\int_{\widetilde{\Sigma}_{\lambda}}\big[\frac1{|x-y|^{n-\alpha}}-\frac1{|x^{\lambda}-y|^{n-\alpha}}\big]{w^{p-1}(y)}|y|^{(n-\alpha)(p-p_\alpha)}dy\\\nonumber
&\leq&C(\Omega)\int_{\Sigma_\lambda}\big[\frac1{|x-y|^{n-\alpha}}-\frac1{|x-y^\lambda|^{n-\alpha}}\big]({w^{p-1}(y)}-{w^{p-1}_\lambda(y)})dy\nonumber\\
&\leq&C(\Omega)\int_{\Sigma^w_\lambda}\big[\frac1{|x-y|^{n-\alpha}}-\frac1{|x-y^\lambda|^{n-\alpha}}\big]({w^{p-1}(y)}-{w^{p-1}_\lambda(y)})dy\nonumber\\
&\leq&C(\Omega)\int_{\Sigma^w_\lambda}\frac1{|x-y|^{n-\alpha}}({w^{p-1}(y)}-{w^{p-1}_\lambda(y)})\nonumber\\
&\leq& (p-1)C(\Omega)\int_{\Sigma^w_\lambda}\frac{w^{p-2}(y)(w(y)-w_\lambda(y))}{|x-y|^{n-\alpha}}dy,\label{UB-4-add}
\end{eqnarray}
where $C(\Omega)>0$. Since $w\in C (\overline {\Omega^*})$, we have $|w| \le c_1$. Hence by using HLS inequality and H\"{o}lder inequality, we have
\begin{eqnarray*}
\| w-w_\lambda \|_{L^r(\Sigma^w_\lambda)}&\leq& (p-1)C(\Omega)(\int_{\Sigma^w_\lambda}(\int_{\Sigma^w_\lambda}\frac{ w^{p-2}(y)[w(y)-w_\lambda(y)]}{|x-y|^{n-\alpha}}dy)^r dx)^{\frac{1}{r}}\nonumber\\
&\leq& (p-1)C(\Omega)c_1^{p-2}\|  w(y)-w_\lambda(y) \|_{L^s(\Sigma^w_\lambda)}\nonumber\\
&\leq&c|\Sigma^w_\lambda|^{\frac1s-\frac1r}\|w-w_\lambda\|_{L^r(\Sigma^w_\lambda)},\label{UB-5-add}
\end{eqnarray*}
where $r$ and $s$ satisfy
\[\frac{1}{r}= \frac{1}{s}-\frac{\alpha}{n}, \ 1\le s<r<\infty, ~\text{and} ~|\Sigma^w_\lambda|=meas(\Sigma^w_\lambda).\]
For example, we can take $r=p_\alpha, s=q_\alpha$.  Now we choose $\delta>0$ and $\varepsilon_1>0$ small enough such that
\[c|\Sigma^w_\lambda|^{\frac1s-\frac1r}\leq\frac{1}{2}.\]
It implies that
\[\|w-w_{\lambda}\|_{L^r(\Sigma^w_\lambda)}\equiv 0.\]
And hence the measure of set $\Sigma^w_\lambda$ must be zero.

We arrive at
\[w_{\lambda}(x)-w(x)\geq 0, \quad ~x\in \Sigma_\lambda, ~\forall\lambda\in [\lambda_0,\lambda_0+\varepsilon_1)\]
since $w$ is continuous. This contradicts to the definition of $\lambda_0$.
Hence, $\lambda_0=\overline{x}$. 

Thus by using the same argument of convex domain, we conclude that there is $\delta>0$ only depending on the domain $\Omega$ such that for any $x\in \{z\in\overline{\Omega^*}: d(z,\partial\Omega^*)<\delta \}$, 
\[w(x)\le C<\infty
\]
uniformly. By using the Kelvin transformation, we conclude that there exists $\delta^\prime>0$ such that $x_q$ must stay out of the region $\{z\in\overline{\Omega}: d(z,\partial\Omega)<\delta^\prime \}$ as $q\to (q_\alpha)^+$.  \hfill$\Box$

To continue, we need the following lemma, which can be seen in \cite{Lieb1983,CLO2006,Li2004}.
\begin{lemma}\label{classification-Whole space}
Let $v\in L_{loc}^{p_\alpha}(\mathbb{R}^n)$ be a positive solution to
 \begin{equation*}\label{HB-5}
v(x)=\int_{\mathbb{R}^n}\frac{v^{p_\alpha-1}(y)}{|x-y|^{n-\alpha}}dy,\quad x\in \mathbb{R}^n,
\end{equation*}
 then $v=c_1(\frac{1}{c_2+|x-y_0|^2})^\frac{n-\alpha}{2}$ for some $c_1, c_2>0,y_0\in\mathbb{R}^n$.
\end{lemma}

\medskip

\noindent
{\bf Proof of (\rmnum2) of Theorem \ref{BL-alpha-small}}.
Notice that $u_q(x)=f_q^{q-1}(x)\in L^p(\Omega)$. Denote $u_q(x_q):=\max\limits_{\overline \Omega} u_q(x)$ and
 $$\mu_q:= u_q^{-\frac {p-2}{\alpha}}(x_q), \, \,  \, \, \Omega_\mu:=\frac{\Omega-x_q}{\mu_q}=\{z \ | \ z=\frac {x-x_q}{\mu_q}, \ \forall x\in \Omega\}.
$$
Define
\begin{equation*}\label{blowup-1}
v_q(z):=\mu_q^{\frac{\alpha}{p-2}} u_q(\mu_q z+x_q), \,  \ \, \, \mbox{for} \, \, z \in \overline\Omega_\mu.
\end{equation*}
Then $v_q$ satisfies
\begin{equation*}\label{sub-equ3}
v_q(z)=\int_{\Omega_\mu}\frac{v_q^{p-1}(y)}{|z-y|^{n-\alpha}}dy, \ \ \ \ z \in \overline\Omega_\mu,
\end{equation*}
and $v_q(0)=1$, $v_q(z) \in (0, 1].$

We claim $v_q(z)$ is equicontinuous on any bounded domain $\widehat{\Omega}\subset \Omega_\mu$ as $q\to (q_\alpha)^+$. In fact, for $z\in\widehat{\Omega}$, we can first write
\[v_q(z)=\int_{\Omega_\mu\setminus B(0,R)}\frac{v_q^{p-1}(y)}{|z-y|^{n-\alpha}}dy+\int_{\Omega_\mu\cap B(0,R)}\frac{v_q^{p-1}(y)}{|z-y|^{n-\alpha}}dy\]
for fixed $R>0$. On one hand, if $z=\frac {x-x_q}{\mu_q}\in \widehat{\Omega}$, then $|x-x_q|\le C \mu_q$ for some $C>0$. Hence for  $z\in \widehat{\Omega}$, $\epsilon>0$ small enough, we have
\begin{eqnarray}\nonumber
&&\int_{\Omega_\mu\setminus B(0,R)}\frac{v_q^{p-1}(y)}{|z-y|^{n-\alpha}}dy\\\nonumber
&=&\mu_q^{\frac{\alpha}{p-2}}\int_{\Omega\setminus B(x_q,R\mu_q)}\frac{u_q^{p-1}(\xi)}{|x-\xi|^{n-\alpha}}d\xi\ \ \ \ (x=x_q+\mu_q z)\\\nonumber
&\le&\mu_q^{\frac{\alpha}{p-2}}(\int_{\Omega}u_q^p(\xi)d\xi)^{\frac{p-1}{p}}
(\int_{\Omega\setminus B(x_q,R\mu_q)}|x-\xi|^{(\alpha-n)p}d\xi)^{\frac{1}{p}}\\
\nonumber
&=&\mu_q^{\frac{\alpha}{p-2}}(\int_{\Omega}u_q^p(\xi)d\xi)^{\frac{p-1}{p}}
(\int_{\Omega\setminus B(x-x_q,R\mu_q)}|\xi|^{(\alpha-n)p}d\xi)^{\frac{1}{p}}\\
\label{small-term}
&\le&C \mu_q^{\frac{\alpha}{p-2}}(R\mu_q)^{\alpha-\frac{n}{q}}=CR^{\alpha-\frac{n}{q}}\mu_q^{\frac{q(n+\alpha)-2n}{q(2-q)}}<\epsilon
\end{eqnarray}
as $q$ close to $q_\alpha$ and $R>0$ large enough, where we use $\alpha<n/q$ as $q$ close to $q_\alpha$. On the other hand, it is easy to see that $\int_{\Omega_\mu\cap B(0,R)}\frac{v_q^{p-1}(y)}{|z-y|^{n-\alpha}}dy\in C^1(\widehat{\Omega})$. Hence
for $z_1,z_2\in \widehat{\Omega}$,
\begin{eqnarray}\nonumber
&&|\int_{\Omega_\mu\cap B(0,R)}\frac{v_q^{p-1}(y)}{|z_1-y|^{n-\alpha}}dy-\int_{\Omega_\mu\cap B(0,R)}\frac{v_q^{p-1}(y)}{|z_2-y|^{n-\alpha}}dy|\\\nonumber
&\le& C\int_{\Omega_\mu\cap B(0,R)}v_q^{p-1}(y)\frac{1}{|\xi-y|^{n-\alpha+1}}dy|z_1-z_2| \\\label{main-term}
&\le& C R^{\alpha-1}|z_1-z_2|,
\end{eqnarray}
where $\xi=tz_1+(1-t)z_2$ for some $t\in(0,1)$. By \eqref{small-term} and \eqref{main-term} we conclude that $v_q(z)$ is equicontinuous on the bounded domain $\widehat{\Omega}\subset \Omega_\mu$ as $q \to (q_\alpha)^+$.

As $q\to (q_\alpha)^+$, since $x_q$ will stay away from $\partial\Omega$, then we have  
  $\Omega_\mu \to \Bbb{R}^n$,
 $v_q(z) \to v(z) \in C_{loc}^\gamma(\Bbb{R}^n)$ uniformly for any $0<\gamma<1$, where $v(z)$ satisfies
$$
v(z)= \int_{\Bbb{R}^n} \frac {v^{p_\alpha-1}(y)}{|z-y|^{n-\alpha}} dy, \ \ \ \ v(0)=1.
$$
We also have $v(z)\in L^{p_\alpha}(\Bbb{R}^n)$ since
\begin{equation}\label{sub-equ4-add1-1}
\int_{\Bbb{R}^n}v^{p_\alpha}(z) dz\le \liminf_{q \to (q_\alpha)^+}\int_{\Omega_\mu} v_q^p(z)dz=\liminf_{q \to (q_\alpha)^+}\mu_q^{\frac{p\alpha}{p-2}-n}\int_{\Omega} u_q^p(x)dx\le C.
\end{equation}
By Lemma \ref{classification-Whole space} and noticing that $\max\limits_{z\in \R^n} v(z)=v(0)$, we know
\begin{equation}\label{sub-equ4-add1}
v(z)=c_1(\frac{1}{c_2+|z|^2})^\frac{n-\alpha}{2}
\end{equation} for some $c_1,c_2>0$ satisfying
\begin{equation}\label{sub-equ5-add1}
c_1\cdot c_2^{\frac{\alpha-n}{2}}=1.
\end{equation}
By \cite{DZ2017} we know $\xi_{\alpha,q_\alpha}(\Omega)=\xi_{\alpha,q_\alpha}(\R^n)$. Then
 $v^{p_\alpha-1}$ must be an extremal function to $\xi_{\alpha,q_\alpha}(\R^n)$ since
\begin{eqnarray*}
\xi_{\alpha,q_\alpha}(\R^n)&\ge&\frac{\int_{\Bbb{R}^n} v^{p_\alpha}(y)dy}{(\int_{\Bbb{R}^n} v^{p_\alpha}(y)dy)^{\frac{2}{q_\alpha}}}=(\int_{\Bbb{R}^n} v^{p_\alpha}(y)dy)^{-\frac{\alpha}{n}}\\
&\ge & (\liminf_{q \to (q_\alpha)^+}\mu_q^{\frac{p\alpha}{p-2}-n}\int_{\Omega} u_q^p(x)dx)^{-\frac{\alpha}{n}}\\
&\ge & (\liminf_{q \to (q_\alpha)^+}\int_{\Omega} u_q^p(x)dx)^{-\frac{\alpha}{n}}\\
&=&\xi_{\alpha,q_\alpha}(\Omega)=\xi_{\alpha,q_\alpha}(\R^n),
\end{eqnarray*}
where we use \eqref{Bas-equ1}, \eqref{sub-equ4-add1-1}, $2<p<p_\alpha$ and $\mu_q\to 0^+$ as $q \to (q_\alpha)^+$.  Therefore
\begin{equation}\label{sub-equ5}
\int_{\Bbb{R}^n} v^{p_\alpha}(y)dy=(\xi_{\alpha,q_\alpha}(\Bbb{R}^n))^{-\frac{n}{\alpha}}.
\end{equation}

 Notice that $p\to (p_\alpha)^-$ is equivalent to $q\to (q_\alpha)^+$.
We claim that
\begin{equation}\label{sub-equ6}
\mu_q^s\to 1, \ \ \mbox{as}\ \ p\to (p_\alpha)^-,
\end{equation}
where $s=p_\alpha-p\to 0^+$. In fact, it is easy to see that $\mu_q^s\le 1$. On the other hand, we know for any $R>0$, $v_q(z) \to v(z) \ \mbox{in}~ B(0,R)\cap \Omega_\mu$ uniformly as $q\to (q_\alpha)^+$. By \eqref{sub-equ5}, for any $\epsilon>0$ small enough, there exists $R>0$ large enough such that
\begin{eqnarray}\label{sub-equ7}
&&(\xi_{\alpha,q_\alpha}(\Omega))^{-\frac{n}{\alpha}}-\epsilon \\\nonumber
&\le& \int_{B(0,R)}v^{p_\alpha}(z)dz=\lim_{p\to (p_\alpha)^-}\int_{B(0,R)\cap \Omega_\mu}v_q^p(z)dz\\\nonumber
&=&\lim_{p\to (p_\alpha)^-}\mu_q^{\frac{p\alpha}{p-2}-n}\int_{B(x_q,R\mu_q)\cap\Omega}u_q^p(x)dx\\\nonumber
&\le& \lim_{p\to (p_\alpha)^-}\mu_q^{\frac{s(n-\alpha)}{p-2}}\int_{\Omega}u_q^p(x)dx\\\nonumber
&=&(\xi_{\alpha,q_\alpha}(\Omega))^{-\frac{n}{\alpha}}\lim_{p\to (p_\alpha)^-}\mu_q^{\frac{s(n-\alpha)}{p-2}},
\end{eqnarray}
where we use \eqref{Bas-equ1} in the last equality. Thus the claim \eqref{sub-equ6} holds by using the arbitrary chosen of $\epsilon$.

To continue, let $w_q(\xi), \xi\in \Omega_\mu^*$ be the Kelvin transform of $v_q$, 
where $\Omega_\mu^*$ is the image of $\Omega_\mu$ under the Kelvin transform. We have
\begin{eqnarray}\label{kelvin-euq1}
w_q(\xi)&=&\int_{\Omega_\mu^*}\frac{w_q^{p-1}(\eta)}{|\xi-\eta|^{n-\alpha}}|\eta|^{(n-\alpha)(p-1)-(n+\alpha)} d\eta\\\nonumber
&\le &C \int_{\Omega_\mu^*}\frac{w_q^{p-1}(\eta)}{|\xi-\eta|^{n-\alpha}}\mu_q^{(n-\alpha)(p-1)-(n+\alpha)}d\eta\\\nonumber
&\le &C \int_{\Omega_\mu^*}\frac{w_q^{p-1}(\eta)}{|\xi-\eta|^{n-\alpha}}d\eta
\end{eqnarray}
by using \eqref{sub-equ6}.

For fixed $1>r>0$ small, which will be determined later, we have $w_q(\xi)\le C(r)$ for $|\xi|\ge \frac{r}{2}$. On the other hand, for $\xi\in B^*(0,\frac{r}{2}):=B(0,\frac{r}{2})\cap\Omega_\mu^*$, by \eqref{kelvin-euq1},
\begin{equation*}\label{kelvin-euq2}
w_q(\xi)\le C\int_{B^*(0,r)}\frac{w_q^{p-1}(\eta)}{|\xi-\eta|^{n-\alpha}}d\eta+C(r).
\end{equation*}
Take $t=\frac{p}{p-1}>q_\alpha$ and $p_1>0$ satisfies $\frac{1}{p_1}=\frac{1}{t}-\frac{\alpha}{n}$, which implies that $p_1>p_\alpha$. Then
\begin{eqnarray*}
&&(\int_{B^*(0,\frac{r}{2})}w_q^{p_1}(\xi)d\xi)^{\frac{1}{p_1}}\\
&\le & C(\int_{B^*(0,r)}(\int_{B^*(0,r)}\frac{w_q^{p-1}(\eta)}{|\xi-\eta|^{n-\alpha}}d\eta)^{p_1}d\xi)^{\frac{1}{p_1}}+C(r)\\
&\le & C\|w_q^{p-1}\|_{L^t(B^*(0,r))}+C(r)\\
&=&C(\int_{B^*(0,r)}(w_q^{p-2}(\eta)w_q(\eta))^t d\eta)^{\frac{1}{t}}+C(r)\\
&\le & C(\int_{B^*(0,r)}w_q^{(p-2)\frac{p_1 t}{p_1-t}}(\eta) d\eta)^{\frac{p_1-t}{p_1 t}}(\int_{B^*(0,r)}w_q^{t\cdot\frac{p_1}{t}}(\eta) d\eta)^{\frac{1}{p_1}}+C(r)\\
&=& C(\int_{B^*(0,r)}w_q^{(p-2)\cdot\frac{n}{\alpha}}(\eta) d\eta)^{\frac{\alpha}{n}}(\int_{B^*(0,r)}w_q^{p_1}(\eta) d\eta)^{\frac{1}{p_1}}+C(r).
\end{eqnarray*}
Notice that $(p-2)\cdot\frac{n}{\alpha}<p$. Then as in \eqref{sub-equ7}, for $\epsilon_0>0$ small enough, there exists $r>0$ small enough such that 
\begin{eqnarray*}
&&\int_{B^*(0,r)} w_q^{(p-2)\cdot\frac{n}{\alpha}}\le C(\int_{B^*(0,r)} w_q^p)^{\frac{(p-2)\frac{n}{\alpha}}{p}}\\
&\le& C(\int_{\Omega_\mu \setminus B(0,\frac{1}{r})} v_q^p)^{\frac{(p-2)\frac{n}{\alpha}}{p}}\le \epsilon_0
\end{eqnarray*}
 uniformly as $q$ close to $q_\alpha$, and then
\[C(\int_{B^*(0,r)}w_q^{(p-2)\cdot\frac{n}{\alpha}}(\eta) d\eta)^{\frac{\alpha}{n}}\le \frac{1}{2}\] as $q$ close to $q_\alpha$. 
Hence
\begin{eqnarray*}
&&(\int_{B^*(0,\frac{r}{2})}w_q^{p_1}(\xi)d\xi)^{\frac{1}{p_1}}\\
&\le& \frac{1}{2}(\int_{B^*(0,r)}w_q^{p_1}(\eta) d\eta)^{\frac{1}{p_1}}+C(r)\\
&\le& \frac{1}{2}(\int_{B^*(0,\frac{r}{2})}w_q^{p_1}(\eta) d\eta)^{\frac{1}{p_1}}+C(r).
\end{eqnarray*}
Then $w_q\in L^{p_1}(B^*(0,\frac{r}{2}))$ and $\|w_q\|_{L^{p_1}(B^*(0,\frac{r}{2}))}\le C(r)$ uniformly as $q\to (q_\alpha)^+$.
Thus similar to Lemma 3.3 in \cite{DZ2017}, we know $w_q\in L^\infty(B^*(0,\frac{r}{2}))$ uniformly as $q\to (q_\alpha)^+$ and thus we conclude.
\hfill$\Box$

\medskip

Now to prove (\rmnum3) of Theorem \ref{BL-alpha-small}, we need the following lemma.
\begin{lemma}\label{est-func-u-small} We have
$u_q^{p-1}(x)u_q(x_q)\to \sigma_{n,\alpha} \delta_{x_0}(x)$ as $p\to (p_\alpha)^-$.
\end{lemma}
\begin{proof}
Take any function $\varphi(x)\in C_0^\infty(\Omega)$.
%Then for any $\epsilon>0$, there exists $r>0$ small enough such that $|\varphi(x)-\varphi(x_0)|<\epsilon$ provided $|x-x_0|<r$.
For any small $r>0$, by (\rmnum2) of Theorem \ref{BL-alpha-small},
\begin{eqnarray*}
&&|\int_{\Omega\setminus B(x_0,r)}u_q^{p-1}(x)u_q(x_q)\varphi(x)dx|\\
&\le& C\int_{\Omega\setminus B(x_0,r)}\mu_q^{\frac{(n-\alpha)(p-1)}{2}}\frac{u_q(x_q)}{|x-x_0|^{(n-\alpha)(p-1)}}dx\\
&\le& C(r)\mu_q^{\frac{(n-\alpha)(p-1)}{2}-\frac{\alpha}{p-2}}\to 0
\end{eqnarray*}
as $p\to (p_\alpha)^-$. On the other hand, by using  \eqref{sub-equ6}, (\rmnum2) of Theorem \ref{BL-alpha-small}, and the Dominated Convergence Theorem, as $p\to (p_\alpha)^-$, we have
\begin{eqnarray*}
&&\int_{B(x_0,r)}u_q^{p-1}(x)u_q(x_q)\varphi(x_0)dx\\
&=& \varphi(x_0)\int_{B(\frac{x_0-x_q}{\mu_q},\frac{r}{\mu_q})}u_q^{p-1}(\mu_q z+x_q)u_q(x_q)\mu_q^n dz\\
&=& \varphi(x_0)\mu_q^{-\frac{p\alpha}{p-2}+n}\int_{B(\frac{x_0-x_q}{\mu_q},\frac{r}{\mu_q})}v_q^{p-1}(z) dz\\
&\to & \varphi(x_0)\int_{\mathbb{R}^n}v^{p_\alpha-1}(z) dz=\varphi(x_0)\sigma_{n,\alpha},
\end{eqnarray*}
since by \eqref{sub-equ4-add1}, \eqref{sub-equ5-add1}, \eqref{sub-equ5},
\begin{eqnarray*}
&&\int_{\mathbb{R}^n}v^{p_\alpha-1}(z) dz=\int_{\mathbb{R}^n}(c_1(\frac{1}{c_2+|x|^2})^{\frac{n-\alpha}{2}})^{\frac{n+\alpha}{n-\alpha}} dz\\
&=&c_2^{\frac{n}{2}}\omega_n\int_0^\infty \frac{r^{n-1}}{(1+r^2)^{\frac{n+\alpha}{2}}}dr=\frac{(\xi_{\alpha,q_\alpha}(\R^n))^{-\frac{n}{\alpha}}\int_0^\infty \frac{r^{n-1}}{(1+r^2)^{\frac{n+\alpha}{2}}}dr }{\int_0^\infty \frac{r^{n-1}}{(1+r^2)^n}dr}\\
&=&(\pi^{\frac{n}{2}}\frac{\Gamma(\frac{\alpha}{2})}{\Gamma(\frac{n+\alpha}{2})})^{\frac{\alpha-n}{\alpha}}=\sigma_{n,\alpha},
\end{eqnarray*}
where $\omega_n$ is the area of the unit sphere in $\R^n$, and 
$$\xi_{\alpha,q_\alpha}(\R^n)=\pi^{(n-\alpha) /2} \frac{\Gamma(\alpha/2)}{\Gamma(n/2+\alpha/2)} \{\frac{\Gamma(n/2)}{\Gamma(n)} \}^{-\alpha/n},$$ see \cite{Lieb1983}.
Thus
\begin{eqnarray*}
&&\int_{B(x_0,r)}u_q^{p-1}(x)u_q(x_q)\varphi(x)dx\\
&=&\int_{B(x_0,r)}u_q^{p-1}(x)u_q(x_q)(\varphi(x_0)+o_r(1))dx\\
&\to& (\varphi(x_0)+o_r(1))\sigma_{n,\alpha}
\end{eqnarray*}
as $p\to (p_\alpha)^-$, where $o_r(1)\to 0$ as $r\to 0$.
Therefore 
\[\int_{\Omega}u_q^{p-1}(x)u_q(x_q)\varphi(x)dx \to \varphi(x_0)\sigma_{n,\alpha}\]
 as $p\to (p_\alpha)^-$. Thus we conclude.
\end{proof}

\noindent
{\bf Proof of (\rmnum3) of Theorem \ref{BL-alpha-small}}. Firstly by using (\rmnum2) of Theorem \ref{BL-alpha-small} it is easy to see that there is one single point $x_0$ such that $x_q\to x_0$ as $q\to (q_\alpha)^+$, up to a subsequence, and
\begin{equation}\label{proof3-equ1}
u_q(x_q)u_q(x)\le C\frac{1}{|x-x_0|^{n-\alpha}},\ \ \ u_q(x)\to 0 \ \mbox{if} \ x\neq x_0,
\end{equation}
as $q\to (q_\alpha)^+$.
For any $\varphi\in C_0^\infty(\Omega)$, it is easy to check that $\int_\Omega\frac{\varphi(x)}{|x-y|^{n-\alpha}}dx$ is continuous in $y\in\overline\Omega$. Then by \eqref{HB-sub-1} and Lemma \ref{est-func-u-small}, for any $\varphi\in C_0^\infty(\Omega)$,
\begin{eqnarray*}
\int_\Omega u_q(x_q)u_q(x)\varphi(x)dx&=&\int_\Omega \varphi(x)dx\int_\Omega\frac{u_q^{p-1}(y)u_q(x_q)}{|x-y|^{n-\alpha}}dy\\
&=&\int_\Omega u_q^{p-1}(y)u_q(x_q)dy\int_\Omega\frac{\varphi(x)}{|x-y|^{n-\alpha}}dx\\
&\to & \sigma_{n,\alpha} \int_\Omega\frac{\varphi(x)}{|x-x_0|^{n-\alpha}}dx,
\end{eqnarray*}
as $q\to (q_\alpha)^+$.  So in order to prove that $u_q(x_q)u_q(x) \to \frac{\sigma_{n,\alpha} }{|x-x_0|^{n-\alpha}}$ pointwise for $x\neq x_0$, it is left to prove that $u_q(x_q)u_q(x)$ is uniformly bounded and equicontinuous in $\Omega \setminus B(x_0,r)$ for any $r>0$ small enough,  as $q\to (q_\alpha)^+$.

In fact, by using \eqref{proof3-equ1}, $u_q(x_q)u_q(x)$ is uniformly bounded in $\Omega \setminus B(x_0,r)$. On the other hand, for any $|x_1-x_2|<\frac{r}{2}, x_1, x_2\in \Omega \setminus B(x_0,r)$, 
\begin{eqnarray}\nonumber
&&|u_q(x_q)u_q(x_1)-u_q(x_q)u_q(x_2)|\\ \nonumber
&=&u_q(x_q)|\int_\Omega \big(\frac{1}{|x_1-y|^{n-\alpha}}-\frac{1}{|x_2-y|^{n-\alpha}}\big)u_q^{p-1}(y)dy|\\ \nonumber
&\le &(n-\alpha)u_q(x_q)\int_\Omega \frac{u_q^{p-1}(y)}{|\xi-y|^{n-\alpha+1}}dy|x_1-x_2| \\ \nonumber
&& \ \ (\xi=tx_1+(1-t)x_2, \ \mbox{for~some}\ t\in (0,1))
\\ \label{proof3-equ2}
&=&(n-\alpha)u_q(x_q)\big(\int_{B(x_0,\frac{r}{2})}+\int_{\Omega\setminus B(x_0,\frac{r}{2})}\big) \frac{u_q^{p-1}(y)}{|\xi-y|^{n-\alpha+1}}dy|x_1-x_2|. 
\end{eqnarray}
By Lemma \ref{est-func-u-small},
\begin{eqnarray}\nonumber
&&u_q(x_q)\int_{B(x_0,\frac{r}{2})}\frac{u_q^{p-1}(y)}{|\xi-y|^{n-\alpha+1}}dy\\\label{proof3-equ3}
&\le & C \int_{B(x_0,\frac{r}{2})} u_q(x_q) u_q^{p-1}(y)dy\to C\sigma_{n,\alpha},
\end{eqnarray}
as $q\to (q_\alpha)^+$. By using (\rmnum2) of Theorem \ref{BL-alpha-small}, 
\begin{eqnarray}\nonumber
&&u_q(x_q)\int_{\Omega\setminus B(x_0,\frac{r}{2})}\frac{u_q^{p-1}(y)}{|\xi-y|^{n-\alpha+1}}dy\\\nonumber
&\le & C u_q(x_q) \int_{\Omega\setminus B(x_0,\frac{r}{2})}  \big(\frac{\mu_q}{\mu_q^2+|x-x_q|^2}\big)^{\frac{(n-\alpha)(p-1)}{2}}\cdot\frac{1}{|\xi-y|^{n-\alpha+1}}dy\\ \nonumber
&\le & C\mu_q^{-\frac{\alpha}{p-2}+\frac{(n-\alpha)(p-1)}{2}}\int_{\Omega\setminus B(x_0,\frac{r}{2})}  \frac{1}{|\xi-y|^{n-\alpha+1}}dy\\  \label{proof3-equ4}
&\le & C\mu_q^{-\frac{\alpha}{p-2}+\frac{(n-\alpha)(p-1)}{2}}\to 0,  
\end{eqnarray}
as $q\to (q_\alpha)^+$.
Thus combining with \eqref{proof3-equ2}, \eqref{proof3-equ3} and \eqref{proof3-equ4}, we conclude that $u_q(x_q)u_q(x)$ is equicontinuous in $\Omega \setminus B(x_0,r)$  as $q\to (q_\alpha)^+$. \hfill$\Box$

\medskip

\section{Blowup behaviour for $\alpha>n$}

\noindent
{\bf Proof of (\rmnum1) of Theorem \ref{BL-alpha-big}}.
 Notice that $f_q=u_q^{p-1}$ is the energy minimizing positive solution to $(\mathcal{P}_q)$,  and $\widehat{\xi}_{\alpha,q}(\Omega)\to \widehat{\xi}_{\alpha,q_\alpha}(\Omega)>0$ as $q\to (q_\alpha)^-$ \cite{DGZ2017}. Then
 \begin{equation}\label{R-Bas-equ1}
||f_q||_{L^q(\Omega)}=(\widehat{\xi}_{\alpha,q} (\Omega))^{\frac{1}{q-2}}\ge\frac{(\widehat{\xi}_{\alpha,q_\alpha}(\Omega))^{\frac{1}{q_\alpha-2}}}{2}>0
 \end{equation}
as $q$ close to $q_\alpha$.

We assume by contradiction that $\min\limits_{\overline \Omega} u_q(x)\ge C_1>0$ uniformly.

If $u_q(x) \le C_2<\infty$ uniformly, then it is easy to see that the $C^1$ norm of $u_q(x)$ is also uniformly bounded by \cite{DGZ2017}. Thus $u_q(x)$ is equicontinuous. Hence $u_q(x) \to u^*(x)>0$ pointwise as $q \to (q_\alpha)^-$. We also have $||f_q||_{L^q(\Omega)}\to ||f^*||_{L^{q_\alpha}(\Omega)}$ as $q\to (q_\alpha)^-$, where $f^*=(u^*)^{p_\alpha-1}$. So $f^*$ is an energy minimizing positive solution to $(\mathcal{P}_{q_\alpha})$, which is of $C^1$ obviously. Then we obtain a contradiction.

If $\max\limits_{\overline \Omega} u_q(x):=u_q(\tilde{x}_q) \to \infty$, then
$$\infty\leftarrow u_q(\tilde{x}_q)=\int_\Omega\frac{u_q^{p-1}(y)}{|\tilde{x}_q-y|^{n-\alpha}}dy\le C<\infty,$$
 as $q\to (q_\alpha)^-$, which again gives a contradiction. 

Thus we conclude that $\min\limits_{\overline \Omega} u_q(x):=u_q(x_q)\to 0$ and
$x_q\to x_0\in\overline\Omega$  as $q\to (q_\alpha)^-$, up to a subsequence.

To prove that $x_q$ will stay away from $\partial\Omega$ as $q\to (q_\alpha)^-$, we consider two cases. For simplicity, in the proof of this part below we write $u$ instead of $u_q$,  $f$ instead of $f_q$.

(\Rmnum1). The domain $\Omega$ is strictly convex.

By using the method of moving planes to integral equation \eqref{HB-sub-1}, which is omitted here since it is standard (see for example \cite{CFY2014}, \cite{DGZ2017}, \cite{DGZ17AM}, \cite{DZ2017}), we can prove that there exist $t_0>0, \alpha>0$ depending on the domain only, such that for every $x\in\partial\Omega$, $u(x-t\nu)$ is decreasing in $t\in [0, t_0]$, where $\nu\in \mathbb{R}^n, |\nu|=1$ satisfying $(\nu,\overrightarrow{n}(x))\ge\alpha$, and $\overrightarrow{n}(x)$ is the unit outer normal of $\Omega$ at the boundary point $x$.

Now as in \cite{Han}, we can prove that there are $\gamma,\delta>0$ only depending on the domain $\Omega$ such that for any $x\in \{z\in\overline{\Omega}: d(z,\partial\Omega)<\delta \}$, there exists a measurable set $\Gamma_x\subset\{z\in\overline{\Omega}: d(z,\partial\Omega)>\delta/2 \}$ satisfying $meas(\Gamma_x)\ge\gamma$ and $u(y)\le u(x)$ for any $y\in\Gamma_x$. As in Theorem \ref{BL-alpha-small}, $\Gamma_x$ again can be taken to be a piece of cone with vertex at $x$. Then for any $x\in \{z\in\overline{\Omega}: d(z,\partial\Omega)<\delta \}$, by \eqref{R-Bas-equ1}, 
\begin{eqnarray*}
u(x)&\ge & \frac{1}{meas(\Gamma_x)}\int_{\Gamma_x} u(y)dy\ge \frac{\gamma^{\frac{1}{q}}}{|\Omega|}(\int_{\Omega} u^p(y)dy)^{\frac{1}{p}}\\
&=&\frac{\gamma^{\frac{1}{q}}}{|\Omega|}(\int_{\Omega} f^q(y)dy)^{\frac{1}{p}}
\ge C>0
\end{eqnarray*}
uniformly, which implies that $x_q$ must stay out of the region $\{z\in\overline{\Omega}: d(z,\partial\Omega)<\delta \}$ as $q\to (q_\alpha)^-$.

(\Rmnum2). The domain $\Omega$ is not necessarily strictly convex.

As in the proof of Theorem \ref{BL-alpha-small}, we use the Kelvin transformation at each boundary point and then apply the method of moving planes. We give the details for the reader's convenience.

Without loss of generality we assume $P=(-1,0,\cdots,0)\in\partial\Omega$ and the ball $B(0,1)$ contacts $P$ from the exterior of $\Omega$. Let $w(x),x\in\Omega^*$ be the Kelvin transform of $u$, where $\Omega^*$ is the image of $\Omega$ under the Kelvin transform.  We also can assume that $x_1=-1$ is the tangent plane of $\partial\Omega^*$ at $P$. Thus $w(x)$ satisfies \eqref{UB-7-add}. Now it is enough to prove that $w(x)$ is decreasing along the $x_1$ direction in a neighbourhood of $P$.

Again as in the proof of Theorem \ref{BL-alpha-small}, assume that we can move the tangent plane at $P$ along the $x_1$ direction to the limiting place $x_1=\overline{x}$, denoted by $T_0$, such that the reflection of $\Omega^*\cap \{x_1<\lambda\}$ with respect to $T_\lambda$ is also a subset of $\Omega^*$, where $T_\lambda=\{x\in \mathbb{R}^n \ | \ x_1=\lambda\}, \lambda\in (-1,\overline{x}]$. 

We denote by $\Omega_1^\prime$ the reflection of $\Omega_1:=\Omega^*\cap \{x_1<\overline{x}\}$ with respect to $T_0$.  Now we apply the method of moving planes to integral equation \eqref{UB-7-add} on $\Omega^\prime:=\Omega_1\cup \Omega_1^\prime$.

For any real number $\lambda\in (-1,\overline{x})$, define 
$x^\lambda, \Sigma_\lambda,\widetilde{\Sigma}_\lambda, \Sigma^C_\lambda$ as in the proof of Theorem \ref{BL-alpha-small}. Set $w_\lambda(x):=w(x^\lambda).$
We complete the proof  in two steps. 

\noindent \textbf{Step 1.}  We show that for $\lambda$ larger than and sufficiently close to $-1$,
\begin{equation}\label{UB-6-add-large}
w(x)-w_\lambda(x)\geq 0,  \quad \forall x\in \Sigma_\lambda, 
\end{equation}
%This provides a start point to move plane $T_\lambda$ along  $x_1$ direction. 
which can be obtained easily by  
\begin{eqnarray*}
\frac{\partial w(x)}{\partial x_1}|_{x_1=-1}
=(\alpha-n)\int_{\Omega^*}\frac{w^{p-1}(y)}{|x-y|^{n-\alpha+2}}(-1-y_1)|y|^{(n-\alpha)(p-p_\alpha)} dy<0
\end{eqnarray*}
since $w(x)\in C^1(\overline{\Omega^*})$.

\noindent \textbf{Step 2.} Plane $T_\lambda$ can be moved continuously towards right to its limiting position $T_0$ as long as inequality \eqref{UB-6-add-large} holds. Thus we  conclude that $w(x)$ is decreasing in the $x_1$ direction for any $x\in\Omega_1$.

Define
\[\lambda_0:=\sup\{\lambda\in (-1,\overline{x})\ | \ w(y)\geq w_\mu(y),\forall y \in\Sigma_\mu, \ \ -1<\mu\leq\lambda\}.\]
We claim that $\lambda_0=\overline{x}$.
For that we assume by contradiction that $\lambda_0<\overline{x}$.

 We first  show that
\[w(x)>w_{\lambda_0}(x),\quad\text{in} ~\Sigma_{\lambda_0}.\]
Hence we have
\[w(x)-w_{\lambda_0}(x)\ge c_1>0,\quad \text{in} ~\Sigma_{\lambda_0-\epsilon_1}
\]for $\epsilon_1>0$ small, which will be determined later.

In fact, as in \eqref{UB-10-add}, we have for $x\in \Sigma_{\lambda_0}$ that
\begin{eqnarray}\label{UB-3-add-large}
&& w(x)-w_{\lambda_0}(x)\\\nonumber
&\geq&\int_{\Sigma^C_{\lambda_0}\backslash\widetilde{\Sigma}_{\lambda_0}}\big[\frac1{|x-y|^{n-\alpha}}-\frac1{|x^{\lambda_0}-y|^{n-\alpha}}\big]{w^{p-1}(y)}|y|^{(n-\alpha)(p-p_\alpha)}dy.
\end{eqnarray}
If there exists some point $\xi\in\Sigma_{\lambda_0}$ such that $w(\xi)= w_{\lambda_0}(\xi)$, then  since $|x-y|>|x^{\lambda_0}-y|$ for $x\in{\Sigma}_{\lambda_0},y\in\Sigma^C_{{\lambda_0}}$, we deduce  from \eqref{UB-3-add-large} that
\[w(y)\equiv \infty,\quad \forall y\in\Sigma^C_{\lambda_0}\backslash\widetilde{\Sigma}_{\lambda_0}.\]
This contradicts to that $w\in C^1(\overline{\Omega^*})$ since $\Sigma^C_{\lambda_0}\backslash\widetilde{\Sigma}_{\lambda_0}$ is not empty.

For fixed small $\delta_1>0$, we choose $\epsilon_1$ small and $\varepsilon \in (0, \epsilon_1)$  such that for any  $\lambda\in [\lambda_0,\lambda_0+\varepsilon)$, there holds
\[w(x)\ge w_{\lambda}(x), \forall x\in \Sigma_{\lambda_0-\varepsilon_1},\]
and
\[
|\frac1{|x-y|^{n-\alpha}}-\frac1{|x^{\lambda}-y|^{n-\alpha}}|\le\delta_1, ~~~~~~~~~ \quad\text{for}~x\in\Sigma_\lambda\backslash\Sigma_{\lambda_0-\varepsilon_1}.
\]

Write
\begin{equation*}
\Sigma^w_\lambda:=\{x\in \Sigma_\lambda |w_\lambda(x)>w(x)\}.
\end{equation*}
Similar to \eqref{UB-4-add}, for any $x\in\Sigma^w_\lambda$,
\begin{eqnarray*}
0&>& w(x)-w_\lambda(x)\\
&\ge&\int_{\Sigma_\lambda}\big[\frac1{|x-y|^{n-\alpha}}-\frac1{|x^\lambda-y|^{n-\alpha}}\big]({w^{p-1}(y)}-{w^{p-1}_\lambda(y)})|y|^{(n-\alpha)(p-p_\alpha)}dy\\
&\ge&\int_{\Sigma_\lambda^w}\big[\frac1{|x-y|^{n-\alpha}}-\frac1{|x^\lambda-y|^{n-\alpha}}\big]({w^{p-1}(y)}-{w^{p-1}_\lambda(y)})|y|^{(n-\alpha)(p-p_\alpha)}dy\\
&\ge&-C(\Omega)\delta_1\int_{\Sigma_\lambda^w}({w^{p-1}(y)}-{w^{p-1}_\lambda(y)})dy,
\end{eqnarray*}
where $C(\Omega)>0$. Since $w\in C^1(\overline{\Omega^*})$,  there exists a positive constant $C_0$ such that $\frac 1{C_0}\le w\le C_0$.   It  follows from the above that
\begin{eqnarray*}
 \int_{\Sigma_\lambda^w}(w_\lambda(x)-w(x))dx
&\le&C(\Omega)\delta_1\int_{\Sigma_\lambda^w}\int_{\Sigma_\lambda^w}({w^{p-1}(y)}-{w^{p-1}_\lambda(y)})dydx\\
&\le&C(\Omega) (1-p)\delta_1\int_{\Sigma_\lambda^w}\int_{\Sigma_\lambda^w}w^{p-2}(y)(w_\lambda(y)-w(y))dydx\\
&\le&C\delta_1(\varepsilon+\varepsilon_1)^n\int_{\Sigma_\lambda^w}(w_\lambda(y)-w(y))dy,
\end{eqnarray*}
which implies that
\[\|w_{\lambda}-w\|_{L^1(\Sigma^w_\lambda)}\equiv 0,\]
 for $\delta_1, \varepsilon, \varepsilon_1>0$ small enough. Hence $\Sigma^w_\lambda$ must have measure zero.

We thus have
\[w(x)-w_{\lambda}(x)\geq 0, \quad~ \text{for~ any} ~x\in \Sigma_\lambda, ~\forall\lambda\in [\lambda_0,\lambda_0+\varepsilon)\]
since $w$ is continuous. This contradicts to the definition of $\lambda_0$.
Hence, $\lambda_0=\overline{x}$.

We hereby  complete the proof by using the same argument as Theorem \ref{BL-alpha-small}.  \hfill$\Box$

To continue, we need the following lemma from \cite{Li2004}.
\begin{lemma}\label{R-classification-Whole space}
Let $v$ be a nonnegative measurable solution to
 \begin{equation*}\label{R-HB-5}
v(x)=\int_{\mathbb{R}^n}\frac{v^{p_\alpha-1}(y)}{|x-y|^{n-\alpha}}dy,\quad x\in \mathbb{R}^n,
\end{equation*}
 then $v=c_1(\frac{1}{c_2+|x-y_0|^2})^\frac{n-\alpha}{2}$ for some $c_1, c_2>0,y_0\in\mathbb{R}^n$.
\end{lemma}

\medskip

\noindent
{\bf Proof of (\rmnum2) of Theorem \ref{BL-alpha-big}}.
Notice that $u_q(x)=f_q^{q-1}(x)\in L^p(\Omega):=\{u| \int_{\Omega}|u|^p dx<\infty\}$,   $p_\alpha<p<0$. Denote $u_q(x_q):=\min\limits_{x\in\overline{\Omega}} u_q(x)\to 0$  as $q \to (q_\alpha)^-$. Let
 $$\mu_q:= u_q^{-\frac {p-2}{\alpha}}(x_q) \, \, \mbox{ and} \, \, \Omega_\mu:=\frac{\Omega-x_q}{\mu_q}=\{z \ | \ z=\frac {x-x_q}{\mu_q} \ \mbox{for} \, x \in \Omega\}.
$$
Define
\begin{equation*}\label{R-blowup-1}
v_q(z):=\mu_q^{\frac{\alpha}{p-2}} u_q(\mu_q z+x_q), \,  \ \, \, \mbox{for} \, \, z \in \overline\Omega_\mu.
\end{equation*}
Then $v_q$ satisfies
\begin{equation*}\label{R-sub-equ3}
v_q(z)=\int_{\Omega_\mu}\frac{v_q^{p-1}(y)}{|z-y|^{n-\alpha}}dy, \ \ \ \ z \in \overline\Omega_\mu,
\end{equation*}
and $v_q(0)=1$, $v_q(z) \ge 1.$

Firstly, we have
\begin{equation}\label{R-sub-equ4}
1=v_q(0)=\int_{\Omega_\mu}v_q^{p-1}(y)|y|^{\alpha-n}dy.
\end{equation}

Claim: We have
\begin{equation}\label{R-sub-equ5}
\liminf\limits_{q\to(q_\alpha)^-}\int_{\Omega_\mu}v_q^{p-1}(y)dy\ge c_0>0.
\end{equation}
Otherwise, if up to a subsequence $\int_{\Omega_\mu}v_q^{p-1}(y)dy\to 0$ as $q \to (q_\alpha)^-$, then for fixed $R_0>0$ we actually can prove that $v_q(z)\to 1$ uniformly for $z\in B(0,R_0)\cap \Omega_\mu$ as $q \to (q_\alpha)^-$, which then gives a contradiction.  In fact for any $\epsilon>0$, there exists $R>0$ large enough and $q$ close to $q_\alpha$, such that
\begin{eqnarray*}
1&\le &v_q(z)\\
&=&\int_{\Omega_\mu\setminus B(0,R)}\frac{v_q^{p-1}(y)}{|z-y|^{n-\alpha}}dy+\int_{B(0,R)}\frac{v_q^{p-1}(y)}{|z-y|^{n-\alpha}}dy\\
&\le& (1+\frac{R_0}{R})^{\alpha-n}\int_{\Omega_\mu\setminus B(0,R)}\frac{v_q^{p-1}(y)}{|y|^{n-\alpha}}dy+(R+R_0)^{\alpha-n}\int_{B(0,R)}v_q^{p-1}(y)dy\\
&\le& (1+\frac{R_0}{R})^{\alpha-n}+(R+R_0)^{\alpha-n}\int_{\Omega_\mu}v_q^{p-1}(y)dy\\
&\le& 1+\frac{\epsilon}{2}+\frac{\epsilon}{2}, \ \mbox{for} \  z\in B(0,R_0)\cap \Omega_\mu \ \mbox{uniformly}.
\end{eqnarray*}
That is, $v_q(z)\to 1, \ z\in B(0,R_0)\cap \Omega_\mu$ uniformly as $q \to (q_\alpha)^-$.

It is also easy to see that $\int_{\Omega_\mu}v_q^{p-1}(y)dy$ is uniformly bounded as $q \to (q_\alpha)^-$ by using \eqref{R-sub-equ4}, which then combining with \eqref{R-sub-equ5} gives that
\begin{equation}\label{R-sub-equ6}
0<C_1(1+|z|)^{\alpha-n}\le v_q(z) \le C_2(1+|z|)^{\alpha-n}, \ \mbox{uniformly for any} \ z.
\end{equation}

Now we claim $v_q(z)$ is equicontinuous on any bounded domain $\widehat{\Omega}\subset \Omega_\mu$ when $q \to (q_\alpha)^-$.  For $z\in\widehat{\Omega}$, we write
\[
v_q(z)=\int_{\Omega_\mu\setminus B(0,R)}\frac{v_q^{p-1}(y)}{|z-y|^{n-\alpha}}dy+\int_{\Omega_\mu\cap B(0,R)}\frac{v_q^{p-1}(y)}{|z-y|^{n-\alpha}}dy.
\]
For $\epsilon>0$ small enough, we have for $z\in\widehat{\Omega}$ that
\begin{eqnarray}\nonumber
&&\int_{\Omega_\mu\setminus B(0,R)}\frac{v_q^{p-1}(y)}{|z-y|^{n-\alpha}}dy\le C\int_{\Omega_\mu\setminus B(0,R)}\frac{v_q^{p-1}(y)}{|y|^{n-\alpha}}dy\\\label{R-small-term}
&\le&C\int_R^\infty r^{(\alpha-n)(p-1)+\alpha-1}dr=C R^{(\alpha-n)(p-1)+\alpha}<\epsilon
\end{eqnarray}
as $q$ close to $q_\alpha$, and $R>0$ large enough. On the other hand, it is easy to see that $\int_{B(0,R)\cap\Omega_\mu}\frac{v_q^{p-1}(y)}{|z-y|^{n-\alpha}}dy\in C^1(\widehat{\Omega})$. Hence
for $z_1,z_2\in \widehat{\Omega}$,
\begin{eqnarray}\nonumber
&&|\int_{B(0,R)\cap\Omega_\mu}\frac{v_q^{p-1}(y)}{|z_1-y|^{n-\alpha}}dy-\int_{B(0,R)\cap\Omega_\mu}\frac{v_q^{p-1}(y)}{|z_2-y|^{n-\alpha}}dy|\\\nonumber
&\le& C\int_{B(0,R)\cap\Omega_\mu}v_q^{p-1}(y)\frac{1}{|\xi-y|^{n-\alpha+1}}dy|z_1-z_2|, \\\label{R-main-term}
&\le& C R^{\alpha-1}|z_1-z_2|,
\end{eqnarray}
where $\xi=tz_1+(1-t)z_2$ for some $t\in(0,1)$. By \eqref{R-small-term} and \eqref{R-main-term} we conclude that $v_q(z)$ is equicontinuous on the bounded domain $\widehat{\Omega}\subset\Omega_\mu$ when $q \to (q_\alpha)^-$.

As $q\to (q_\alpha)^-$, since $x_q$ will stay away from $\partial\Omega$, we have $\Omega_\mu \to \Bbb{R}^n$,
 $v_q(z) \to v(z) \in C(\Bbb{R}^n)$, where $v(z)$ satisfies
$$
v(z)= \int_{\Bbb{R}^n} \frac {v^{p_\alpha-1}(y)}{|z-y|^{n-\alpha}} dy, \ \ \ \ v(0)=1.
$$
By Lemma \ref{R-classification-Whole space} and noticing that $\min\limits_{z\in \R^n} v(z)=v(0)$, we know
\begin{equation}\label{R-sub-equ6-add1}
v(z)=c_1(\frac{1}{c_2+|z|^2})^\frac{n-\alpha}{2}
\end{equation} for some $c_1,c_2>0$ satisfying
\begin{equation}\label{R-sub-equ7-add1}
c_1\cdot c_2^{\frac{\alpha-n}{2}}=1.
\end{equation}
By \cite{DGZ2017} we know $\widehat{\xi}_{\alpha,q_\alpha}(\Omega)=\widehat{\xi}_{\alpha,q_\alpha}(\R^n)$. Then we know
 $v^{p_\alpha-1}$ is an extremal function to $\widehat{\xi}_{\alpha,q_\alpha}(\R^n)$ since
\begin{eqnarray*}
\widehat{\xi}_{\alpha,q_\alpha}(\R^n)&\le&\frac{\int_{\Bbb{R}^n} v^{p_\alpha}(y)dy}{(\int_{\Bbb{R}^n} v^{p_\alpha}(y)dy)^{\frac{2}{q_\alpha}}}=\big(\int_{\Bbb{R}^n} v^{p_\alpha}(y)dy\big)^{-\frac{\alpha}{n}}\\
&=& \big(\lim_{q \to (q_\alpha)^-}\int_{\Omega} v_q^p(z)dz\big)^{-\frac{\alpha}{n}}\\
&=& \big(\lim_{q \to (q_\alpha)^-}\mu_q^{\frac{p\alpha}{p-2}-n}\int_{\Omega} u_q^p(x)dx\big)^{-\frac{\alpha}{n}}\\
&\le &\big(\lim_{q \to (q_\alpha)^-}\int_{\Omega} u_q^p(x)dx\big)^{-\frac{\alpha}{n}}\\
&=&\widehat{\xi}_{\alpha,q_\alpha}(\Omega)=\widehat{\xi}_{\alpha,q_\alpha}(\R^n),
\end{eqnarray*}
where we use \eqref{R-sub-equ6}, \eqref{R-Bas-equ1}, $p_\alpha<p<0$ and $\mu_q\to 0^+$ as $q \to (q_\alpha)^-$. Then
\begin{equation}\label{R-sub-equ7}
\int_{\Bbb{R}^n} v^{p_\alpha}(y)dy=(\widehat{\xi}_{\alpha,q_\alpha}(\Bbb{R}^n))^{-\frac{n}{\alpha}}.
\end{equation}

To continue, let $w_q$ be the Kelvin transform of $v_q$, that is,
$$w_q(\xi)=\frac{1}{|\xi|^{n-\alpha}}v_q(\frac{\xi}{|\xi|^2}), \ \ \ \ \xi\in\overline{\Omega_\mu^*},$$
which satisfies
\begin{equation*}
w_q(\xi)=\int_{\Omega_\mu^*}\frac{w_q^{p-1}(\eta)}{|\xi-\eta|^{n-\alpha}}|\eta|^{(n-\alpha)(p-1)-(n+\alpha)} d\eta,
\end{equation*}
where $\Omega_\mu^*$ is the image of $\Omega_\mu$ under the Kelvin transform.

For fixed $r>0$ small, which will be determined later, we have $w_q(\xi)\ge C(r)$ for $|\xi|\ge \frac{r}{2}$.

For a fixed $r_0>0$, we know that $v_q(x)\to v(x)=c_1(\frac{1}{c_2+|z|^2})^\frac{n-\alpha}{2}$ uniformly in $B_{1/r_0}(0)\setminus B_{1/(2r_0)}(0)$ as  $q\to (q_\alpha)^-$. Then for $\xi\in (B_{2r_0}(0)\setminus B_{r_0}(0))\cap \Omega_\mu^*$,
\begin{eqnarray*}
w_q(\xi)&\le&|\xi|^{\alpha-n}(v(\frac{\xi}{|\xi|^2})+o(1))\\
&\le&C r_0^{\alpha-n}(C+o(1))\\
&\le&C(r_0).
\end{eqnarray*}
Then for any $\xi\in B_r(0)\cap \Omega_\mu^*$ with $r>0$ small enough, 
\begin{equation*}
w_q(\xi)\ge C\int_{(B_{2r_0}(0)\setminus B_{r_0}(0))\cap \Omega_\mu^*}\frac{w_q^{p-1}(\eta)}{|\xi-\eta|^{n-\alpha}}d\eta \ge C(r_0)
\end{equation*}
uniformly as $q\to (q_\alpha)^-$. Thus we conclude.
\hfill$\Box$

\begin{remark}
As in \eqref{sub-equ6}, we also can prove 
\begin{equation}\label{R-sub-equ8}
\mu_q^s\to 1, \ \ \mbox{as}\ \ p\to (p_\alpha)^+,
\end{equation}
where $s=p-p_\alpha\to 0^+$. In fact, it is easy to see that $\mu_q^s\le 1$. By \eqref{R-sub-equ6} and \eqref{R-sub-equ7},
\begin{eqnarray*}
&&(\widehat{\xi}_{\alpha,q_\alpha}(\Omega))^{-\frac{n}{\alpha}}=\lim_{p\to (p_\alpha)^+}\int_{\Omega}u_q^p(x)dx\\
&=&\lim_{p\to (p_\alpha)^+}\mu_q^{n-\frac{p\alpha}{p-2}}\int_{\Omega_\mu}v_q^p(z)dz=\lim_{p\to (p_\alpha)^+}\mu_q^{\frac{s(n-\alpha)}{p-2}}\int_{\Omega_\mu}v_q^p(z)dz\\
&=&\lim_{p\to (p_\alpha)^+}\mu_q^{\frac{s(n-\alpha)}{p-2}}\int_{\R^n}v^{p_\alpha}(z)dz=(\widehat{\xi}_{\alpha,q_\alpha}(\Omega))^{-\frac{n}{\alpha}}\lim_{p\to (p_\alpha)^+}\mu_q^{\frac{s(n-\alpha)}{p-2}},
\end{eqnarray*}
where we use \eqref{R-Bas-equ1} in the first equality. Thus \eqref{R-sub-equ8} holds.
\end{remark}

\medskip

Now to prove (\rmnum3) of Theorem \ref{BL-alpha-big}, we need the following lemma.
\begin{lemma}\label{est-func-u-big} We have
$u_q^{p-1}(x)u_q(x_q)\to \sigma_{n,\alpha} \delta_{x_0}(x)$ as $p\to (p_\alpha)^+$.
\end{lemma}
\begin{proof} Take any function $\varphi(x)\in C_0^\infty(\Omega)$.
For any $r>0$ small, by (\rmnum2) of Theorem \ref{BL-alpha-big},
\begin{eqnarray*}
&&|\int_{\Omega\setminus B(x_0,r)}u_q^{p-1}(x)u_q(x_q)\varphi(x)dx|\\
&\le& C\int_{\Omega\setminus B(x_0,r)}\mu_q^{\frac{(n-\alpha)(p-1)}{2}}\frac{u_q(x_q)}{|x-x_0|^{(n-\alpha)(p-1)}}dx\\
&\le& C\mu_q^{\frac{(n-\alpha)(p-1)}{2}-\frac{\alpha}{p-2}}\to 0
\end{eqnarray*}
as $p\to (p_\alpha)^+$. On the other hand, by using \eqref{R-sub-equ6-add1}, \eqref{R-sub-equ7-add1}, \eqref{R-sub-equ7} and \eqref{R-sub-equ8}, similar to Lemma \ref{est-func-u-small} we have
\begin{equation*}
\int_{B(x_0,r)}u_q^{p-1}(x)u_q(x_q)\varphi(x_0)dx\to \varphi(x_0)\int_{\mathbb{R}^n}v^{p_\alpha-1}(z) dz =\varphi(x_0)\sigma_{n,\alpha},
\end{equation*}
as $p\to (p_\alpha)^+$. 
Thus
\begin{eqnarray*}
\int_{B(x_0,r)}u_q^{p-1}(x)u_q(x_q)\varphi(x)dx\to (\varphi(x_0)+o_r(1))\sigma_{n,\alpha}
\end{eqnarray*}
as $p\to (p_\alpha)^+$, where $o_r(1)\to 0$ as $r\to 0$.
Hence 
\[\int_{\Omega}u_q^{p-1}(x)u_q(x_q)\varphi(x)dx \to \varphi(x_0)\sigma_{n,\alpha}\]
 as $p\to (p_\alpha)^+$.
\end{proof}

\noindent
{\bf Proof of (\rmnum3) of Theorem \ref{BL-alpha-big}}. From (\rmnum2) of Theorem \ref{BL-alpha-big} it is easy to see that there is one single point $x_0$ such that $x_q\to x_0$, up to a subsequence, and
\begin{equation*}%\label{R-proof3-equ1}
u_q(x_q)u_q(x)\ge C|x-x_0|^{\alpha-n},\ \ \ u_q(x)\to \infty \ \mbox{if} \ x\neq x_0,
\end{equation*}
as $q\to (q_\alpha)^-$.
For any $\varphi\in C_0^\infty(\Omega)$, it is also easy to check that $\int_\Omega\frac{\varphi(x)}{|x-y|^{n-\alpha}}dx$ is continuous in $y\in\overline\Omega$. Then by \eqref{HB-sub-1} and Lemma \ref{est-func-u-big}, for any $\varphi\in C_0^\infty(\Omega)$,
\begin{eqnarray*}
\int_\Omega u_q(x_q)u_q(x)\varphi(x)dx&=&\int_\Omega \varphi(x)dx\int_\Omega\frac{u_q^{p-1}(y)u_q(x_q)}{|x-y|^{n-\alpha}}dy\\
&=&\int_\Omega u_q^{p-1}(y)u_q(x_q)dy\int_\Omega\frac{\varphi(x)}{|x-y|^{n-\alpha}}dx\\
&\to &  \sigma_{n,\alpha} \int_\Omega\frac{\varphi(x)}{|x-x_0|^{n-\alpha}}dx,
\end{eqnarray*}
 as $q\to (q_\alpha)^-$. 
  So to prove that $u_q(x_q)u_q(x) \to   \frac{\sigma_{n,\alpha}}{|x-x_0|^{n-\alpha}}$ pointwise for $x\neq x_0$, it is enough to prove that $u_q(x_q)u_q(x)$ is uniformly bounded and equicontinuous in $\Omega \setminus B(x_0,r)$ for any $r>0$ small enough, as $q\to (q_\alpha)^-$.

In fact, by using \eqref{R-sub-equ6} and \eqref{R-sub-equ8},
\begin{eqnarray*}
&&u_q(x_q)u_q(x)=\mu_q^{-\frac{2\alpha}{p-2}} v_q(z) \le C \mu_q^{-\frac{2\alpha}{p-2}}(1+|z|)^{\alpha-n}  \ \ \ (z=\frac{x-x_q}{\mu_q})\\
&=&C \mu_q^{-\frac{2\alpha}{p-2}}(1+|\frac{x-x_q}{\mu_q}|)^{\alpha-n}\le C \mu_q^{-\frac{2\alpha}{p-2}+n-\alpha}\le C
\end{eqnarray*}
uniformly in $\Omega \setminus B(x_0,r)$ as $q\to (q_\alpha)^-$. On the other hand, similar to the proof of (\rmnum3) of Theorem \ref{BL-alpha-small}, by using Lemma \ref{est-func-u-big} and (\rmnum2) of Theorem \ref{BL-alpha-big},
we can prove that $u_q(x_q)u_q(x)$ is equicontinuous in $\Omega \setminus B(x_0,r)$  as $q\to (q_\alpha)^-$. \hfill$\Box$

\medskip

 \vskip 1cm
\noindent {\bf Acknowledgements}\\
\noindent The author would like to thank Professor Meijun Zhu for many helpful discussions and valuable comments during the preparation of this manuscript. The project is supported by  the
National Natural Science Foundation of China (Grant No. 11571268).





%% bibliography--------------------------------------------------------------------
%\begin{center}
%\section*{References}

\end{document}